\documentclass[paper=a4, fontsize=11pt]{scrartcl} 

\usepackage[T1]{fontenc} 
\usepackage{fourier} 
\usepackage[english]{babel} 
\usepackage{amsmath,amsfonts,amsthm} 
\usepackage{amssymb}
\usepackage{verbatim}
\usepackage{lipsum} 
\usepackage{hyperref}

\usepackage{sectsty} 
\allsectionsfont{\centering \normalfont\scshape} 

\usepackage{fancyhdr} 
\numberwithin{equation}{section} 
\numberwithin{figure}{section} 
\numberwithin{table}{section} 

\newtheorem{thm}{Theorem}[section]
\newtheorem{cor}[thm]{Corollary}

\newtheorem{quest}[thm]{Question}
\theoremstyle{definition}

\newtheorem{exmp}[thm]{Example}

\theoremstyle{remark}
\newtheorem{rem}[thm]{Remark}

\DeclareMathOperator{\Hom}{Hom}

\DeclareMathOperator{\Gal}{Gal}

\DeclareMathOperator{\Tr}{Tr}

\DeclareMathOperator{\End}{End}
\DeclareMathOperator{\Frob}{Frob}

\DeclareMathOperator{\Pic}{Pic}
\DeclareMathOperator{\Spec}{Spec}

\DeclareMathOperator{\rk}{rank}

\DeclareMathOperator{\cha}{char}

\DeclareMathOperator{\Fal}{Fal}
\DeclareMathOperator{\Sing}{Sing}

\newcommand{\horrule}[1]{\rule{\linewidth}{#1}} 

\title{	
\normalfont \normalsize 
\textsc{} \\ [25pt] 
\horrule{0.5pt} \\[0.4cm] 
\huge Decomposability and Mordell-Weil ranks of Jacobians using Picard numbers  \\ 
\horrule{2pt} \\[0.5cm] 
}

\author{Soohyun Park} 

\date{\normalsize\today} 

\begin{document}

\maketitle 

\begin{abstract}
	\noindent We use methods for computing Picard numbers of reductions of K3 surfaces in order to study the decomposability of Jacobians over number fields and the variance of Mordell-Weil ranks of families of Jacobians over different ground fields. For example, we look at surfaces whose Picard numbers jump in rank at all primes of good reduction using Mordell-Weil groups of Jacobians and show that the genus of curves over number fields whose Jacobians are isomorphic to a product of elliptic curves satisfying certain reduction conditions is bounded. The isomorphism result addresses the number field analogue of some questions of Ekedahl and Serre on decomposability of Jacobians of curves into elliptic curves.
\end{abstract}



\section{Introduction}


Given a family of curves, we can associate a family of Jacobians to it. Then, it is natural to consider the following question:


\begin{quest}\label{Q1}
	How does the Mordell-Weil rank of a family of Jacobians vary as we change the base field?
\end{quest}

By Mordell-Weil ranks of families, we mean the Mordell-Weil rank of a Jacobian over a function field. Other than being an interesting question to think about in general, studying the function field case of a problem on properties of rational points has been useful for understanding what happens in other global fields such as number fields and some other questions that we may be interested in. One topic we will study is the relationship between Picard numbers of products of curves and Jacobians of certain curves. \\

The other question that we will spend most of our time studying is the following geometric question:


\begin{quest}\label{Q2}
	Is there an upper bound on the genus of curves over a fixed number field which have Jacobians isomorphic to a product of elliptic curves?
\end{quest}

This is motivated by some questions raised by Ekedahl and Serre \cite{ES} for curves over $\mathbb{C}$. More specifically, they asked the following questions on the decomposability of Jacobians of smooth projective curves over $\mathbb{C}$ into products of elliptic curves up to isogeny or isomorphism:

\begin{enumerate}
	\item Does there exist a Jacobian of a curve of genus $g$ which is decomposable into a product of elliptic curves up to isogeny for \emph{every} $g > 0$?
	
	\item Is the set of such $g$ bounded?
	
	\item Does there exist $g > 3$ such that some curve of genus $g$ has its Jacobian \emph{isomorphic} to a product of elliptic curves?
\end{enumerate}

There has been some previous work on the second problem (e.g. \cite{P}, \cite{R}) which has been successful in producing a large number of explicit decompositions coming from representations of the endomorphism algebra of the Jacobian of a curve induced by a finite group action. Another perspective on the problem has to do with unlikely intersections involving (weakly) special subvarieties of Shimura varieties (e.g. \cite{LZ}). One example of the tools used is analogues of positivity results from complex algebraic geometry in Arakelov theory such as slope inequalities. Motivated by the second approach, Chen, Lu, and Zuo \cite{CLZ} prove a finiteness result for Jacobians of curves over number fields in the self-product case assuming the Sato-Tate conjecture. This is interesting since some of the authors working on curves over $\mathbb{C}$ seem to guess that the genus is unbounded in the second question posed by Ekedahl and Serre. \\


In general, products of elliptic curves tend to large Picard numbers. This is especially true for self-products of elliptic curves $E^g$, which have the maximal Picard number $g^2$ if $E$ is a CM elliptic curve and $\frac{g(g + 1)}{2}$ otherwise \cite{HL}. This would mean that Jacobians isogenous to a product of elliptic curves would have a large Picard number. However, this is something possible. In fact, Beauville \cite{Bea} gives examples of Jacobians of curves of genus $4$ and $10$ which are \emph{isomorphic} to a product of isogenous CM elliptic curves. This gives a positive answer to the third question of Ekedahl and Serre stated at the beginning since it examples of curves of genus $g > 3$ whose Jacobians are isomorphic to a product of elliptic curves. Given the second question of Ekedahl and Serre, one may ask whether the genus of such curves is bounded as in Question \ref{Q2}. \\

Instead of attempting to directly compute the Picard number over the ground field or its algebraic closure, the main tool which we will use to try to understand Question \ref{Q2} is computation of Picard numbers of reductions of a product of curves $C \times C$ for a curve $C$ over a number field. In Theorem \ref{gbound}, we use this to obtain a result analogous to that of Chen, Lu, and Zuo \cite{CLZ} without assuming the Sato-Tate conjecture. The distribution modulo reductions at different primes was recently studied for K3 surfaces in recent work of Costa, Elsenhans, and Jahnel \cite{CEJ}. Although this has not been done yet for products of curves, many of the main observations made here also seem to carry over to our case since the Tate conjecture holds for products of curves. \\

The connection of this work with Question \ref{Q1} comes from a formula of Ulmer (Theorem 6.4 of \cite{U}) which expresses the Mordell-Weil rank of certain Jacobians over function fields in terms of homomorphisms between two Jacobians and some terms that depend on the geometry of the construction. More specifically, the behavior of both the Picard number and the Mordell-Weil rank of Jacobians in these families mainly depend on the how the same object under reduction at good primes. \\

The methods used by Costa, Elsenhans, and Jahnel \cite{CEJ} which are used to understand this jumping behavior will be described in further detail in Section 2. Next, we will apply the observations made in \cite{CEJ} to a product of curves in Section 3. Finally, Section 4 will contain applications of these methods to the the questions we mentioned at the beginning (\ref{Q1}, \ref{Q2}). Section 4.1 will give examples of applications of this work to Question \ref{Q1} as we consider the Mordell-Weil rank over function fields over an algebraically closed field and its residue fields. Question \ref{Q2} and other properties of decomposable Jacobians over number fields will be considered in Section 4.2 mostly through the proof of Theorem \ref{gbound}.

\section{Distribution of Picard numbers of reductions of surfaces}

In this section, we will describe some of the tools used by Costa, Elsenhans, and Jahnel in \cite{CEJ} to analyze Picard numbers of reductions of $K3$ surfaces following the first section of their paper \cite{CEJ}. \\

Some of the techniques originate from explicit computations of Picard numbers of K3 surfaces over number fields using reductions modulo good primes in work of van Luijk \cite{vL} which was later further extended by Elsenhans in Jahnel in various directions. \\

Given a surface $S$ over a number field $K$, it is well-known that $\rk \Pic S_{\overline{\mathbb{F}}_\mathfrak{p}} \ge \rk \Pic S_{\overline{K}}$. However, it is not as clear how to determine exactly when the rank actually jumps after specialization. \\

Some of the main tools they use to try to measure this are certain characters depending on the action of the (absolute) Frobenius map on $\Pic S_{\overline{K}}$, $H^2_{\'et}(S_{\overline{K}}, \mathbb{Q}_l(1))$, and transcendental part $T$ of $H^2_{\'et}(S_{\overline{K}}, \mathbb{Q}_l(1))$. \\

The character which is most directly connected to their main results on jumping of Picard numbers after reduction is the \emph{jump character} \[ \left( \frac{\Delta_{H^2}(S)\Delta_{\Pic}(S)}{\mathfrak{p}} \right), \]

where $K(\Delta_{\Pic}(S))$ is the minimal field of definition of $\bigwedge^{\max} \Pic S_{\overline{K}}$ and $\Delta_{H^2}(S)$ is defined similarly. \\

The relation of these objects to the results of \cite{CEJ} on jumping of Picard numbers of K3 surfaces after reduction is summarized in the following theorem:

\begin{thm} \label{jumping} (Costa, Elsenhans, Jahnel \cite{CEJ})
	Let $K$ be a number field and $S$ be a K3 surface over $K$. Moreover, let $\mathfrak{p} \subset \mathcal{O}_K$ be a prime of good reduction and residue characteristic $\ne 2$. 
	
	\begin{enumerate}
		\item The following two questions hold:
		
		\begin{itemize}
			\item $\det(\Frob_{\mathfrak{p}} : H^2_{\'et}(S_{\overline{K}}, \mathbb{Q}_l(1)) ) = \left( \frac{\Delta_{H^2}(S)}{\mathfrak{p}} \right)$
			
			\item $\det(\Frob_{\mathfrak{p}} : T ) = \left( \frac{\Delta_{H^2}(S)\Delta_{\Pic}(S)}{\mathfrak{p}} \right)$
		\end{itemize}
		
		\item If $\rk \Pic S_{\overline{K}}$ is even, then \[ \left( \frac{\Delta_{H^2}(S)\Delta_{\Pic}(S)}{\mathfrak{p}} \right) = -1 \Rightarrow \rk \Pic S_{\overline{\mathbb{F}}_{\mathfrak{p}}} \ge \rk \Pic S_{\overline{K}} + 2. \]
			
		In other words, the Picard number jumps at the primes $\mathfrak{p}$ which are inert in $K(\sqrt{\Delta_{H^2}(S)\Delta_{\Pic}(S)})$.
		
	\end{enumerate}
\end{thm}

\section{Applications to products of curves}


We can show that the main theoretical observation of \cite{CEJ} on the distribution of Picard numbers of K3 surfaces over number fields under specialization (Proposition 2.4.2) carries over to products of curves. 

\begin{thm}
	Let $C$ be a smooth projective curve over a number field $K$ and $\mathfrak{p} \subset \mathcal{O}_K$ be a prime of good reduction. If $\det\Frob_{\mathfrak{p}}|_T = -1$, then $\rk \Pic (C \times C)_{\overline{\mathbb{F}}_{\mathfrak{p}}} \ge \rk \Pic (C \times C)_{\overline{K}} + 2$. 
\end{thm}

\begin{proof}
	Most of the proof is same as that of Proposition 2.4.2 since the Tate conjecture holds for products of curves. So, it suffices to show the Picard number of $C \times C$ over $\mathbf{F}_{\mathfrak{p}}$ is always even. \\
	
	As in \cite{Y}, let $f(C, T) = \prod_{i = 1}^{2g} (1 - \alpha_i T)$ be the Weil polynomial for $C$. Then, we have that \[ P_2(C \times C, T) = (1 - qT)^2 \prod_{i = 1}^{2g} \prod_{j = 1}^{2g} (1 - \alpha_i \alpha_j T) \]
	
	is the reverse characteristic polynomial the action of $\Frob_{\mathfrak{p}}$ on $H^2_{\'et}(C \times C, \mathbb{Q}_l)$ for $l \ne |\mathbf{F}_{\mathfrak{p}}|$. So, the eigenvalues are actually the $\alpha_i\alpha_j$. Recall that the Picard number is the multiplicity of $q$ as a root of $P_2(C \times C, T)$ since the Tate conjecture holds for this surface. Since $\alpha_i\alpha_j = \alpha_j\alpha_i$ for $i \ne j$, we have that these eigenvalues always have even multiplicity. Now suppose that $\alpha_i^2 = q$. Then, we also have $\alpha_{2g - i}^2 = q$ since $\alpha_i \alpha_{2g - i} = q$ by the Weil conjectures. Thus, the $\alpha_i$ such that $\alpha_i^2 = q$ always come in pairs and the Picard number must be even.
\end{proof}

\begin{cor}
	If $\det\Frob_{\mathfrak{p}}|_T = -1$, the density of jump primes is at least $\frac{1}{2}$.
\end{cor}

\begin{proof} 
	This follows from the Chebotarev density theorem and the definition of the jump character since the inert primes have density $\frac{1}{2}$ in a quadratic extension.
\end{proof}

To find when $\det\Frob_{\mathfrak{p}}|_T = -1$, it suffices to find the discriminant of the Picard representation in our case. Since the leading coefficient in the characteristic polynomial above is $q^{2g + 2}$ and twisting by $1$ divides each of the eigenvalues by $q$, we have that $\left( \frac{\Delta_{H^2}(C \times C)}{\mathfrak{p}} \right) = 1$ for all good primes $\mathfrak{p}$. The means that the jump character is $\left( \frac{\Delta_{\Pic}(C \times C)}{\mathfrak{p}} \right)$ and it it suffices to find the field of definition of $\bigwedge^{\max} \Pic(C \times C)_{\overline{K}}$ in order to determine the jump character. If the determinant of the action of $\Frob_{\mathfrak{p}}$ on $\Pic(C \times C)$ is $-1$, $K(\sqrt{\Delta_{\Pic}(C \times C)})$ is actually a quadratic number field. Otherwise, $K(\sqrt{\Delta_{\Pic}(C \times C)}) = K$. 

\begin{rem}
	Recall that $\Pic(C \times C)_k \cong (\Pic(C))^2 \times \End_k J(C)$ if $C$ has a $k$-rational point (which is always true when $k = \overline{K}$). The elements of $\mathbb{Z}^2$ are classes corresponding to horizontal and vertical divisors. So, it suffices to compute $\End_k J(C)$ in order to find the minimal field of definition of $\bigwedge^{\max} \Pic(C \times C)_{\overline{K}}$. By recent work in \cite{CMSV}, such an algorithm actually exists for a curve of arbitrary genus $g$. To find the jump character, it suffices to find the determinant of the action of the Galois group of the field of definition on the endomorphism ring.
\end{rem}

We can give examples of products of curves $C \times C$ of when the jump character is trivial and nontrivial. 

\begin{exmp}\label{trivial} (Trivial character) 
	By \cite{UZ}, we have that $\End_{\overline{K}} J(C) = \mathbb{Z}[\zeta_p]$ for $C : y^p = f(x)$ with $p \ge 5$ an odd prime and $f$ a polynomial of degree $p$ with distinct roots such that the Galois group of the splitting field is $S_p$ or $A_p$. \\
	
	Then $\Pic(C \times C)_{\overline{K}} = \mathbb{Z} \times \mathbb{Z}[\zeta_p]$. Note that the field of definition is $\mathbb{Q}(\zeta_p)$. To find the field of definition of $\bigwedge^{\max} \Pic(C \times C)_{\overline{K}}$, we find the determinant of the action of multiplication by $\zeta_p$ on $\Pic(C \times C)_{\overline{K}} = \mathbb{Z} \times \mathbb{Z}[\zeta_p]$. Then, the $(p - 1) \times (p - 1)$ matrix of this action with respect to the basis $\zeta_p, \ldots, \zeta_p^{p - 1}$ is \[ \begin{pmatrix} 
	0 & \cdots & 0 & -1 \\ 1 & \cdots & 0 & -1 \\ \vdots & \cdots & \vdots & \vdots \\ 0 & \cdots & 1 & -1
	\end{pmatrix}. \]
	
	Since the determinant of this matrix is $1$, the field of definition is just $\mathbb{Q}$ and the jump character is trivial. 
	
\end{exmp}

\begin{exmp} (Nontrivial character)
	On the other hand, the character is nontrivial in the main higher genus example of \cite{CMSV} (Example 8.2.1). In this case, the endomorphism algebra is $\mathbb{Q} \times \mathbb{Q}(\sqrt{17})$. Since the jump character is nontrivial, our earlier density result applies.
\end{exmp}


\section{Mordell-Weil ranks and decomposable Jacobians}

\subsection{Mordell-Weil ranks and endomorphism rings of Jacobians}
Endomorphism rings of Jacobians are also connected to Mordell-Weil ranks of Jacobians of certain curves. The most well-known example of ths is the Shioda-Tate formula.

\begin{thm}(Shioda-Tate\cite{U})
	Let $\mathcal{X}$ be a smooth and proper surface over $k$ and $X \longrightarrow \Spec K$ be the generic fiber of $\pi$. \\
	
	Then, we have \[ \rk NS(\mathcal{X}) = \rk MW(J_X) + 2 + \sum_v (f_v - 1), \] where the sum is over closed points of $\mathbb{P}_k^1$ and $f_v$ is the number of irreducible components in the fiber of $\pi$ over $v$.
\end{thm}

Using this formula along with a construction of Berger \cite{Ber} which generates families of curves with Jacobians where BSD holds, Ulmer \cite{U} gives a new more explicit formula which gives ranks over certain types of function fields (Theorem 6.4 of \cite{U}, Theorem 9.4 of \cite{UCJ}).

\begin{thm}\label{rkformula}(Ulmer \cite{U})
	Assume that $k$ is algebraically closed. Choose smooth proper irreducible curves $\mathcal{C}$ and $\mathcal{D}$ over $k$ and non-constant separable rational functions $f : \mathcal{C} \longrightarrow \mathbb{P}_k^1$ and $g : \mathcal{D} \longrightarrow \mathbb{P}_k^1$ satisfying condition (4.1.1) on p. 7 of \cite{U}. \\
	
	Let $X$ be the smooth proper model of \[ \{ f - tg = 0 \} \subset \mathcal{C} \times_k \mathcal{D} \times_k \Spec K \] over $K := k(t)$ constructed in Section 4 of \cite{U}. If $(d, \cha k) = 1$, then \[ \rk MW(J_d) = \rk \Hom_k\left( J_{\mathcal{C}'_{d/e_{d, f}}}, J_{\mathcal{D}'_{d/e_{d, g}}} \right)^{\mu_{d/(e_{d, f} e_{d, g})}} - c_1(d) + c_2(d), \] where $K_d = k(t^{\frac{1}{d}})$. The superscript $\mu_r$ means that the homomorphisms commute with the action of $\mu_r$ on the two Jacobians, where $\mu_r$ denotes the $r^{\text{th}}$ roots of unity in $k$. \\
	
	If $e_{d, f} = e_{d, g} = 1$, the rank formula simplifies to \[ \rk MW(J_d) = \rk \Hom_k(J_{\mathcal{C}_d}, J_{\mathcal{D}_d})^{\mu_d} - c_1(d) + c_2(d). \]
\end{thm}

The terms in this formula depend on the homomorphisms between Jacobians of certain curves and the terms $c_1(d)$ and $c_2(d)$, which depend on geometric properties of these curves. 

\begin{rem}
	The assumption that $k$ is algebraically closed is not strictly necessary. It can be removed given some adjustments on the parameters. See Remark 6.5 of \cite{U} for more details.
\end{rem}


Although exactly determining the endomorphism ring is difficult, there is an algorithm for computing the endomorphism ring of the Jacobian of a curve of arbitrary genus \cite{CMSV}. Nevertheless, we will focus on cases where computation of the endomorphism algebra is feasible as noted by Ulmer \cite{UCJ}. More specifically, we will consider superelliptic curves whose Jacobians have products of cyclotomic fields as their endomorphism algebras.

\begin{exmp}\label{UZ}
	Let $C_{f, d}$ be the smooth projective curve over $k$ with $z^d = f(x)$ as an affine model, where $\deg f = m$ with distinct roots and $d = p^r$ for some prime $p$. Let $J_{f, q} = J(C_{f, q})$ with $q = p^r$. According to \cite{UZ}, $J_{f, q}$ is isogenous to \[ \prod_{i = 1}^r J^{(f, p^i)}, \] where $J^{(f, q)}$ is the kernel of the homomorphism of Jacobians induced by the projection $C_{f, q} \longrightarrow C_{f, \frac{q}{p}}$ defined by the map $(x, z) \mapsto (x, z^p)$. By Theorem 2.5.1 of \cite{UZ}, $\End J^{(f, q)} = \mathbb{Z}[\zeta_q]$ if $m \ge 5$ and $\Gal(f) = S_m$ or $A_m$.  \\
	
	Fixing $f$ and $m$ satisfying the conditions above, we find that we can use the simplified rank formula from Theorem \ref{rkformula}. Taking $d = q$ in Theorem \ref{rkformula}, we find that any element of the endomorphism ring actually commutes with the action of $\mu_d$ and we just get the original endomorphism ring. \\ 
	
	Since $c_1(d)$ and $c_2(d)$ depend entirely on geometric properties, jumping after reduction mod $p$ depends entirely on the endomorphism term. Rewriting this in terms of Picard numbers of $C_{f, d} \times C_{f, d}$, we can see that the characters from \cite{CEJ} can be used to describe where the Mordell-Weil rank of $J(C_{f, d})$ jumps. \\
	
	To determine the jump character from \cite{CEJ} as in Section 3, we can multiply the determinants of the action each component of the product and see whether it is $1$ or $-1$. 
\end{exmp}

Although an algorithm exists (in principle) to study the jumping of Picard numbers, using this to give a precise general answer is difficult. However, there is still an interesting example of a surface where the Picard number jumps after reduction at every good prime.

\begin{exmp}
	To give an example of such a surface, we use the construction from Theorem 7.5 of \cite{U}. 
	
	\begin{thm}\label{jumpexp}(Ulmer \cite{U})
		Take $\mathcal{C} = \mathcal{D} = \mathbb{P}_k^1$, $f(x) = x(x - 1)$, and $g(y) = \frac{y^2}{1 - y}$ in Theorem \ref{rkformula}. Note that $e_{d, f} = e_{d, g}$ in this case. \\
		
		If $\cha k = 0$, then $\rk X_d(K_d) = 0$ for all $d$. If $\cha k = p > 0$, then $\rk X_d(K_d)$ is unbounded as $d$ varies. Suppose that $k = \mathbb{F}_q$ and $d = p^n + 1$. Then, \[ \rk X_d(K_d) \ge \sum_{e|d, e > 2} \frac{\varphi(e)}{o_e(q)} \ge \frac{p^n - 1}{2n}. \] 
		
		If $k$ contains the $d^{\text{th}}$ roots of unity, then $\rk X_d(K_d) = p^n - 1$ if $p$ is odd and $\rk X_d(K_d) = p^n$ if $p = 2$. 
	\end{thm}
	
	If $d = 1$, then the first term of Theorem \ref{rkformula} is exactly the ring of homomorphisms and we have the simplified formulas for $\rk X(K)$ above ($K = k(t)$). We also have that $c_1(d) = c_2(d)$ for all values of $d$ (see section 7.2 of \cite{U}). Then, the jumping in the Mordell-Weil rank after reduction is exactly due to the change in rank of the homomorphism ring after reduction since \[ NS(C \times D) \cong \mathbb{Z}^2 \times \Hom_k(J(C), J(D)) \] if $C$ and $D$ have $k$-rational points. This means that we can translate this into jumping of the Picard number in this case. In general, we can use the same idea if $\frac{d}{e_{d, f} e_{d, g}} = 1$.
	
	
	
	\begin{rem}
		There are some more general results which can help provide heuristics about where to expect the Mordell-Weil rank of curves in Theorem \ref{rkformula} to jump and how much they should jump.
		
		\begin{enumerate}
			\item Work on a conjecture of Murty and Patankar \cite{MP} on the splitting of simple abelian varieties over a number field can be used to guess how often the rank of the endomorphism algebra should jump. More specifically, they conjectured that the set of good places where the reduction of a simple abelian variety is also simple has density 1 if and only if the geometric endomorphism ring is commutative. Although Zywina \cite{Z} has shown that this conjecture is false without possibly replacing the ground field by a finite extension, he has proven results that point toward the general conjecture of Murty and Patankar (see Theorem 1.2 and Corollary 1.3 of \cite{Z}). \\
			
			For example, the Jacobians of the curves discussed in Example \ref{UZ} are simple abelian varieties which we expect to stay simple after reduction by primes in a set of density $1$ if they satisfy additional conditions given in Corollary 1.3 of \cite{Z}. This is consistent with the jump character of the self-product of the curves being trivial in Example \ref{trivial}. It would be interesting to compare how consistent the jump character is with results related to the Murty-Patankar conjecture in general. 
			
			\item For curves in Theorem \ref{rkformula}, the variance of the Mordell-Weil rank mainly varies on how the rank of the homomorphism ring varies. Since $NS(C \times D) \cong \mathbb{Z}^2 \times \Hom_k(J(C), J(D))$, how much the rank of the homomorphism ring jumps is equivalent to finding how much the Picard number jumps. Consider the case where $C = D$ (i.e. the endomorphism ring). The analysis of Picard numbers used to prove Theorem \ref{gbound} can be used to give heuristics on the distribution of the Picard numbers of reductions of a product of curves $C \times C$. For a ``random'' hyperelliptic curve $C$ of genus $g$, the trace of a random matrix in the group $USp(2g)$ of $2g \times 2g$ unitary symplectic matrices gives upper bounds on the distribution of Picard numbers of reductions of $C \times C$ mod good primes. \\
			
			By the proof of Theorem \ref{gbound}, we have that $\Tr \varphi_2^* = 2q + F_1^2$ and that the Picard number over $\overline{\mathbb{F}}_q$ is the multiplicity of $q$ as an eigenvalue of the action of Frobenius on $H^2_{\'et}((C \times C)_{\overline{\mathbb{F}_q}}, \mathbb{Q}_l)$. This means that the Picard number is bounded above by $\frac{F_1^2}{q} = \left( \frac{F_1}{\sqrt{q}} \right)^2$. By work of Katz and Sarnak \cite{KS}, the distribution of $\frac{F_1}{\sqrt{q}}$ is the trace of a random matrix in the group $USp(2g)$ of $2g \times 2g$ unitary symplectic matrices. \\
		
			Thus, this kind of analysis can also be used to control jumps of Picard numbers of products of curves $C \times C$ after reduction modulo good primes. From work of Ulmer \cite{U}, this also applies to changes of Mordell-Weil ranks of Jacobians of certain curves after specialization.
			 
		\end{enumerate}
	\end{rem}
	
\end{exmp}

\subsection{Decomposable Jacobians and Picard numbers of reductions}

Assuming the Sato-Tate conjecture and building on the work of Kukulies \cite{K}, Chen, Lu, and Zuo \cite{CLZ} prove that the genus of smooth projective curves over number fields of bounded degree whose Jacobians are isogenous to a self-product a single elliptic curve is bounded (see Theorem 1.2 of \cite{CLZ}). By considering the Picard numbers of self-products of such curves, we are able prove a result on decomposable Jacobians over number fields without assuming the Sato-Tate conjecture. This depends on showing that decomposability has implications for point counting. \\


%

	\begin{thm} \label{gbound}
	\hfill
	\begin{enumerate}
		\item 
		\begin{itemize}
			\item Suppose that the Jacobian $J(C)$ of a smooth projective curve $C$ of genus $g$ over a number field $K$ is isogenous to $E^g$ for some elliptic curve $E$ over $K$ with supersingular reduction at a prime of norm $\le N$. Then, $g \le G(K, N)$ for some constant $G(K, N)$ depending on $K$ and $N$. 
			
			\item Suppose that $C$ and $J(C)$ have good reduction at the same places. If we bound the Faltings height of $E$ above by $h$, the degree of the isogeny by $d$, and the degree of the number field by over which $C$ is defined by $m$, then $g \le G(h, d, m)$ for a constant $G(h, d, m)$ depending on $h$, $d$, and $m$.
			
			\item Suppose that $[K : \mathbb{Q}]$ is odd or $K$ has a real place and $J(C)$ is isogenous to $E^g$ over $K$ for some elliptic curve $E$ over $K$. Then, there are infinitely many primes $\mathfrak{p} \subset \mathcal{O}_K$ such that the reduction of $C$ mod $\mathfrak{p}$ is maximal after a field extension of degree $\le 3$. 
		\end{itemize}

		
		\item A curve over $\mathbb{C}$ is said to have \emph{many automorphisms} if it cannot be deformed nontrivially with its automorphism group (see p. 2 of \cite{MP1} and the definition on p. 66 of \cite{Po})  Let $K = \mathbb{Q}(i)$ and $C : y^2 = f(x)$ be a hyperelliptic smooth projective curve over $K$ of given genus $g \ge 25$ with many automorphisms with a cyclic automorphism group. Then, $C$ is not isogenous to $E^g$ for any elliptic curve $E$ with CM by an order in $K$. 
		
		
		\item  Let $C$ be a curve over $\overline{\mathbb{Q}}$ with $J(C)$ isogenous to $E^g$ over $\overline{\mathbb{Q}}$ for some elliptic curve $E$ over $\overline{\mathbb{Q}}$ with $j$-invariant $j$. Suppose that $E$ has CM by an order in an imaginary quadratic field $K = \mathbb{Q}(\sqrt{d})$. Let $L = \mathbb{Q}(j)$ and suppose that $C$ and $E$ have $L$ as a minimal field of definition. Then, there are infinitely many primes $\mathfrak{p} \subset \mathcal{O}_L$ such that the reduction of $C$ mod $\mathfrak{p}$ is maximal or minimal in the following cases:
		
		\begin{itemize}
			\item $K = \mathbb{Q}(\sqrt{d})$ such that $d$ is not a quadratic residue mod $p$ for infinitely many primes $p$ such that $p \equiv 1 \pmod {12}$ and $L = \mathbb{Q}(\zeta_{2^k})$ for some $k \ge 2$ and contains $K$, where $\zeta_m$ is a primitive $m^{\text{th}}$ root of unity. 
			
			\item $K$ is the same as above and $L = K$.
		\end{itemize}
		
	\end{enumerate}
\end{thm}
	
	\begin{rem} \hfill
		\begin{enumerate} 
			\item In part 2, we can say something similar about any hyperelliptic curve with many automorphisms of odd degree with sufficiently large genus if $C$ and $E$ are taken to be over a number field where the curve $y^2 = x^{2g + 1} - x$ is isomorphic to the curve $y^2 = x^{2g + 1} + x$. Then, we can use Theorem 3.6 of \cite{KNT} and the fact that our elliptic curve must have CM. The results of Theorem 3.6 and congruence conditions from Deuring's criterion to get incompatible congruence conditions (see proof of part 2).
			
			\item We can obtain a result similar to part 3 and actually get infinitely many primes where the reductions of curves of the form $C : x^n + y^m = 1$ mod $\mathfrak{p} \subset \mathbb{Q}(i)$ are \emph{maximal} for infinitely many primes $\mathfrak{p}$ in $\mathbb{Q}(i)$ if $E$ has CM by an order in $\mathbb{Q}(i)$ and $C$ is taken to be a curve over $\mathbb{Q}(i)$ using Theorem 5 of \cite{TT}. 
			
			\item Parts 2 gives restrictions on jumping of Picard numbers of $C \times C$ for certain curves $C$ with Jacobians isogenous to powers of elliptic curves and part 3 gives an example where it jumps at infinitely many primes for certain curves $C$.
		\end{enumerate}
	\end{rem}
	
\begin{proof}[Proof of Theorem \ref{gbound}]
	\hfill
	\begin{enumerate}
		\item 
		Suppose that we have a curve $C$ over $K$ such that $J(C)$ is isogenous to $E^g$ over $K$ for some $E$ satisfying the conditions above. Let $\mathfrak{p} \subset \mathcal{O}_K$ be a prime where $E$ has supersingular reduction. Let $q$ be the size of the residue field $\mathbb{F}_{\mathfrak{p}}$. Since the Tate conjecture holds for a product of curves, the Picard number is the multiplicity of $q$ as an eigenvalue of the action of $\Frob_{\mathfrak{p}}$ on $H_{\'et}^2 ((C \times C)_{\overline{\mathbb{F}}_{\mathfrak{p}}}, \mathbb{Q}_l)$ at a good prime $\mathfrak{p}$ if $l \ne \cha \mathbb{F}_{\mathfrak{p}}$. This means that we can interpret the Picard number in terms of point counting on $C$. \\
		
		Our method of analyzing the Picard number follows the specialization method outlined in van Luijk's thesis \cite{vL}. Let $F_1$ be the trace of the action of $\Frob_{\mathfrak{p}}$ on $H^1_{\'et}(C_{\overline{\mathbb{F}_{ \mathfrak{p}}} }, \mathbb{Q}_l)$ and $f(C, T)$ be the numerator of the Weil zeta function $Z(C, T)$ as in \cite{Y}. Let $X = C \times C$. Write $P_i(X, T)$ for the reverse characteristic polynomial of the action of the absolute Frobenius map $\Frob_{\mathfrak{p}}$ on $H^i_{\'et}(X_{\overline{\mathbb{F}_{ \mathfrak{p}}} }, \mathbb{Q}_l)$. Then, we have $P_0(X, T) = 1 - T$ and $P_4(X, T) = 1 - q^2 T$ by definition. \\
		
		By Lemma 1.1 of \cite{Y}, we can use a K\"unneth formula for \'etale cohomology to find that $P_1(X, T) = f(C, T)^2$ and $P_3(X, T) = f(C, qT)^2$. Writing $q = p^r$ for a prime $p$, let $\varphi$ be the $r^{\text{th}}$ power of the absolute Frobenius map for $X$ and $\varphi_i^*$ for the induced map on $H^i_{\'et}(X_{\overline{\mathbb{F}_{ \mathfrak{p}}} }, \mathbb{Q}_l)$. Since the $P_i$ give \emph{reverse} characteristic polynomials of the action of the Frobenius map, the coefficient of the linear term multiplied by $-1$ gives us the trace of the action of the this map. This means that $\Tr \varphi_0^* = 1$, $\Tr \varphi_1^* = 2F_1$, $\Tr \varphi_3^* = 2qF_1$, and $\Tr \varphi_4^* = q^2$.  \\
		
		To compute the remaining trace $\Tr \varphi_2^*$, we use the Grothendieck-Lefschetz trace formula, which gives \[ \#X(\mathbb{F}_q) = \sum_{i = 0}^4 (-1)^i \Tr \varphi_i^* \] in our case. Since $X = C \times C$ and $\#C(\mathbb{F}_q) = q + 1 - F_1$, this means that
		
		\begin{align*}
		(q + 1 - F_1)^2 &= \Tr \varphi_0^* - \Tr \varphi_1^* + \Tr \varphi_2^* - \Tr \varphi_3^* + \Tr \varphi_4^* \\
		&= 1 - 2F_1 + \Tr \varphi_2^* - 2q F_1 + q^2 \\
		\Rightarrow \Tr \varphi_2^* &= (q + 1 - F_1)^2 - 1 + 2F_1 + 2q F_1 - q^2 \\
		&= (q + 1)^2 - 2F_1(q + 1) + F_1^2 - 1 + 2F_1 + 2q F_1 - q^2 \\
		&= q^2 + 2q + 1 - 2q F_1 - 2F_1 + F_1^2 - 1 + 2F_1 + 2q F_1 - q^2 \\
		&= 2q + F_1^2.
		\end{align*}
		
		Since the Tate conjecture holds for $X = C \times C$, the Picard number is the multiplicity of $q$ as an eigenvalue of the action of $\Frob_{\mathfrak{p}}$ on $H^2(X_{\overline{\mathbb{F}}_q}, \mathbb{Q}_l)$. We also have that the maximal Picard number $4g^2 + 2$ is attained after some finite extension since $E$ has supersingular reduction at $\mathfrak{p}$. In order for the maximal Picard number $4g^2 + 2$ to be attained, we need $F_1^2 = 4g^2q \Rightarrow |F_1| = 2g\sqrt{q}$. \\
		
		This means that $\#C(\mathbb{F}_q)$ must either reach the upper or lower bound given by the Hasse-Weil bound over some finite extension (of degree $\le 3$ -- see below) in order for this to happen. As mentioned in \cite{Y}, this assumes $q$ is a large even power since we take the Frobenius endomorphisms on $J(C)$ to be rational. When the genus of $C$ is sufficiently large, a result of Ihara \cite{Iha} (e.g. see Theorem 2.6 in \cite{V}) says that the upper bound cannot be obtained if the genus is large compared to $q$. As observed by Lauter \cite{L}, the lower bound given by the Hasse-Weil bound is negative when $q$ is small compared to the genus $g$ (i.e. the same situation as above). \\
		
		So far, we have shown that the genus is bounded given a \emph{specific} prime of $K$ where $J(C)$ has good reduction and $E$ has supersingular reduction for some elliptic curve $E$. Since the number of prime ideals of a given number field with norm $\le N$ is bounded, we still get a bound for primes of norm $\le N$ for a fixed number field $K$. \\
		
		Then, the result after assuming a bound on the Faltings height follows from the proof of Proposition 4.4 of \cite{CLZ}. \\
		
		If $\cha \mathbb{F}_{\mathfrak{p}} >> 0$, then the Hasse-Weil bound is attained after an extension of degree $\le 2$ by Theorem 2.5 of \cite{Y} since $C \times C$ attains has the maximal Picard number $4g^2 + 2$ in this case. Combining this with the fact that there are infinitely many supersingular primes in the cases listed above (\cite{E}, \cite{E2}) gives us infinitely many primes where the third statement holds.
		
				\item Since $g \ge 25$ and $C$ has many automorphisms as a curve over $\mathbb{C}$ with cyclic automorphism group, we have that $C$ is the curve $y^2 = x^{2g + 1} - 1$. By Theorem 3.6 of \cite{KNT} (also see Theorem 2 of \cite{Va}), the reduction of $C$ mod $\mathfrak{p}$ lying above $p$ in $K$ is maximal or minimal over $\mathbb{F}_{p^2}$ if $p \equiv -1 \pmod {4g}$ for sufficiently large $p$ since we are considering a curve of a fixed genus $g$. Recall that Deuring's criterion states that an elliptic curve with CM by an order in an imaginary quadratic field $K$ has supersingular reduction at a good prime $\mathfrak{p}$ lying over $p$ if and only if only one prime lies above $p$ in $K$. For example, this includes inert primes. In our case, this occurs when $p \equiv 3 \pmod 4$. On the other hand, we have ordinary reduction if $p \equiv 1\pmod 4$. Combining this with the condition for $p \equiv -1 \pmod{4g}$, we find that there are infinitely many primes of $K$ where $J(C)$ should be supersingular and $E$ has ordinary reduction by Dirichlet's theorem on primes in arithmetic progressions. Thus, $J(C)$ cannot be isogenous to $E^g$ over $\mathbb{Q}(i)$.  
				
				\item In \cite{KP}, Karemaker and Pries classified supersingular abelian varieties over finite fields according to the maximality or minimality of point counts after a finite extension. They put the supersingular abelian varieties $A$ over finite fields $\mathbb{F}_q$ in the following categories:
				
				\begin{itemize}
					\item If each of the $\mathbb{F}_q$-twists of $A$ has a finite extension of $\mathbb{F}_q$ where it attains the Hasse-Weil upper bound, then $A$ is \emph{fully maximal}.
					
					\item If none of the $\mathbb{F}_q$-twists of $A$ have this property, then $A$ is \emph{fully minimal}.
					
					\item  If some (but not all) of the $\mathbb{F}_q$-twists of $A$ attain the Hasse-Weil upper bound over some finite extension, then $A$ is \emph{mixed}.
				\end{itemize}
				
				Let $\mathfrak{p}$ be a prime where $E$ has supersingular reduction and $p = \cha \mathbb{F}_{\mathfrak{p}}$. By Theorem 6.3 of \cite{KP} and Theorem 4.6(3) of \cite{KP}, $E^g$ is mixed and $J(C)$ must be mixed if it is isogenous to $E^g$ over $\mathbb{F}_{\mathfrak{p}}$ and the $j$-invariant of the specialization $\tilde{E}/\overline{\mathbb{F}}_p$ is not in $\mathbb{F}_p$. Recall from the previous part that the the reduction of the curve $C$ mod $\mathfrak{p}$ must attain the upper or lower bound given by the Hasse-Weil bound after a finite extension. \\
				
				Let the \emph{period} be the minimal degree of such an extension as in \cite{KP}. From part $1$, this is at most $3$ if $p >> 0$. By Deuring's criterion, an elliptic curve $E$ with CM by $K$ has supersingular reduction at $\mathfrak{p}$ if and only if there is a unique prime of $K$ lying above $p$. For example, this includes the inert primes. By the conditions listed in the above, there are infinitely many inert $p$ in $K$ such that $p \equiv 1 \pmod{12}$. Suppose that $L = \mathbb{Q}(\zeta_{2^k})$. Since $L/\mathbb{Q}$ is a Galois extension, we have that $[L : \mathbb{Q}] = efg$, where $e$ is the ramification index, $f := [\mathbb{F}_{\mathfrak{p}} : \mathbb{F}_p ]$ is the residue field extension (also called inertia degree), and $g$ is the number of primes lying above $p$. We would like to show that there are infinitely many $p$ where $f$ is even and $g \ne [L : \mathbb{Q}]$. Note that $f$ is even if $f \ne 1$ in our case since $f$ divides $[L : \mathbb{Q}] = \varphi(2^k) = 2^{k - 1}$ and $k \ge 2$. \\
				
				We also have that $p$ is unramified in $L$ for $p \ne 2$ (see Proposition 2.3 of \cite{W}). This means that $e = 1$ for $p \equiv 1 \pmod {12}$. By Theorem 2.13 of \cite{W}, $p$ does not split completely as long as $p \not\equiv 1 \pmod {2^k}$. Fixing a suitable residue mod $2^k$, it follows from Dirichlet's theorem on prime numbers in arithmetic progressions that there are infinitely many primes $p \equiv 1 \pmod {12}$ such that $p$ does not split completely. Comparing this with the conditions on $r$ and $p$ in Table 1 below Lemma 6.1 of \cite{KP}, we see that the reduction of $C$ is fully maximal or fully minimal at these primes. \\
				
				
				Suppose that $\mathbb{Q}(\sqrt{d})$ is an imaginary quadratic field such that $d$ is not a quadratic residue mod $p$ for infinitely many primes such that $p \equiv 1 \pmod {12}$ and $L = K$. Then, the reduction of $C$ is maximal or minimal at these primes by Table 1 below Lemma 6.1 of \cite{KP}. 
				

			\end{enumerate} 
		
	\end{proof}

Using the fact that a Jacobian has an irreducible principal polarization (see Lemma 2.2 of \cite{La}), we obtain the following corollary:

\begin{cor}\label{isombound}
	The genus $g$ of a smooth projective curve $C$ over a number field $K$ such that $J(C)$ isomorphic to a product of elliptic curves is bounded in each of the conditions listed in part 1 of Theorem \ref{gbound}.
\end{cor}

\begin{proof}
	Note that $J(C)$ is isomorphic to a product of elliptic curves only if it is isogenous to a self-product of elliptic curves $E^g$. Otherwise, the only polarizations of $J(C)$ would be reducible although the canonical polarization from the theta divisor is irreducible (see Lemma 2.2 of \cite{La}). This reduces our problem to Theorem \ref{gbound}.
\end{proof}

\end{document}